\documentclass[11pt,a4paper]{article}

\usepackage{epsf,epsfig,amsfonts,amsgen,amsmath,amstext,amsbsy,amsopn,amsthm
}
\usepackage{ebezier,eepic}
\usepackage{color}
\usepackage{multirow}
\newtheorem{theorem}{Theorem}[section]

\newtheorem{lemma}{Lemma}
\newtheorem{false statement}{False statement}

\theoremstyle{definition}

\newtheorem{claim}{Claim}

\newtheorem{problem}{Problem}

\newcommand{\bL}{\b{\textit{L}}}
\newcommand{\bN}{\b{\textit{N}}}

\newcommand{\bix}{\textbf{\textit{x}}}

\begin{document}

\title{\bf\Large Spectral analogues of Erd\H{o}s' and Moon-Moser's theorems on Hamilton cycles}

\date{}

\author{Binlong Li\thanks{Department of Applied Mathematics,
Northwestern Polytechnical University, Xi'an, Shaanxi 710072, P. R.
China. European Centre of Excellence NTIS, Plze\v n, 30614, Czech
Republic. This author is
supported by the project NEXLIZ -- CZ.1.07/2.3.00/30.0038.}~ and Bo
Ning\thanks{Corresponding author. Center for Applied Mathematics,
Tianjin University, Tianjin, 300072, P. R. China. email:
bo.ning@tju.edu.cn.}}

\maketitle

\begin{abstract}
In 1962, Erd\H{o}s gave a sufficient condition for Hamilton cycles
in terms of the vertex number, edge number, and minimum degree of graphs which
generalized Ore's theorem. One year
later, Moon and Moser gave an analogous result for Hamilton cycles in
balanced bipartite graphs. In this paper we present the spectral
analogues of Erd\H{o}s' theorem and Moon-Moser's theorem,
respectively. Let $\mathcal{G}_n^k$ be the class of non-Hamiltonian
graphs of order $n$ and minimum degree at least $k$. We determine
the maximum (signless Laplacian) spectral radius of graphs in
$\mathcal{G}_n^k$ (for large enough $n$), and the minimum (signless
Laplacian) spectral radius of the complements of graphs in
$\mathcal{G}_n^k$. All extremal graphs with the maximum (signless
Laplacian) spectral radius and with the minimum (signless Laplacian)
spectral radius of the complements are determined, respectively. We
also solve similar problems for balanced bipartite graphs and the
quasi-complements.
\medskip


\end{abstract}

\section{Introduction}

For a graph $G$, we denote by $n(G)$ the order of $G$, by $e(G)$ the
edge number of $G$, by $\delta(G)$ the minimum degree of $G$ and by
$\omega(G)$ the clique number of $G$. For two disjoint graphs $G_1$
and $G_2$, the \emph{union} of $G_1$ and $G_2$, denoted by
$G_1+G_2$, is defined as $V(G_1+G_2)=V(G_1)\cup V(G_2)$ and
$E(G_1+G_2)=E(G_1)\cup E(G_2)$; and the \emph{join} of $G_1$ and
$G_2$, denoted by $G_1\vee G_2$, is defined as $V(G_1\vee
G_2)=V(G_1)\cup V(G_2)$, and $E(G_1\vee G_2)=E(G_1+G_2)\cup
\{xy:x\in V(G_1),y\in V(G_2)\}$. The union of $k$ disjoint copies of
the same graph $G$ is denoted by $kG$.

Let $G$ be a graph, $A$ be the adjacency matrix of $G$ and $D$ be the
degree matrix of $G$. Let $Q=A+D$ be the signless Laplacian
matrix of $G$. The \emph{spectral radius} of $G$, denoted by
$\rho(G)$, is the largest value of eigenvalues of $A$. The
\emph{signless Laplacian spectral radius} of $G$, denoted by $q(G)$,
is the largest value of eigenvalues of $Q$.

A graph $G$ is \emph{Hamiltonian} (\emph{traceable}) if it contains
a Hamilton cycle (Hamilton path), i.e., a cycle (path) containing
all vertices of $G$. Determining whether a given graph is
Hamiltonian or not is an old problem in graph theory. This problem
was proved to be an NP-hard problem \cite{Ka}. For a long time,
graph theorists have been interested in finding sufficient conditions
of Hamilton cycles.
\subsection{Hamiltonicity and traceability of graphs}

In extremal graph theory, a natural problem on Hamilton cycles is,
how many edges can guarantee the existence of a Hamilton cycle in a
graph of order $n$? Ore \cite{O} showed that the condition $e(G)\geq
\binom{n-1}{2}+2$ is the answer.

\begin{theorem}[Ore \cite{O}]\label{ThOr}
Let $G$ be a graph of order $n$. If $$e(G)>\binom{n-1}{2}+1,$$ then
$G$ is Hamiltonian.
\end{theorem}

Note that the graph obtained from $K_{n-1}$ by adding a pendent edge
has $\binom{n-1}{2}+1$ edges but is non-Hamiltonian. This example
shows the condition in Theorem \ref{ThOr} is the best possible.
However, the extremal graph has a vertex of degree 1, and is
trivially non-Hamiltonian. In 1962, Erd\H{o}s \cite{E}
generalized Ore's theorem by imposing minimum degree as a new
parameter.

\begin{theorem}[Erd\H{o}s \cite{E}]\label{ThEr}
Let $G$ be a graph of order $n$. If $\delta(G)\geq k$, where $1\leq
k\leq(n-1)/2$, and
$$e(G)>\max\left\{\binom{n-k}{2}+k^2,\binom{\lceil
(n+1)/2\rceil}{2}+\left\lfloor
\frac{n-1}{2}\right\rfloor^2\right\},$$ then $G$ is Hamiltonian.
\end{theorem}

By Dirac's theorem \cite{D} which states that every graph of
order $n\geq 3$ is Hamiltonian if $\delta(G)\geq n/2$, we can see
the condition $k\leq(n-1)/2$ in Theorem \ref{ThEr} is reasonable.
Furthermore, by simple computation, we know that if $n\geq 6k-2$,
then $\binom{n-k}{2}+k^2\geq\binom{n-l}{2}+l^2$, where
$l=\lfloor(n-1)/2\rfloor$. So Theorem \ref{ThEr} implies that every
graph of order $n\geq 6k-2$ with $\delta(G)\geq k$ and
$e(G)>\binom{n-k}{2}+k^2$, is Hamiltonian.

In this paper, we define, for $1\leq k\leq(n-1)/2$,
$$L^k_n=K_1\vee(K_k+K_{n-k-1}) \mbox{ and } N^k_n=K_k\vee(K_{n-2k}+kK_1).$$
Note that $L^1_{n}=N^1_{n}$. We remark that the graph $N^k_n$
($n\geq 6k-2$) and the graph $N_n^{\lfloor(n-1)/2\rfloor}$ ($n\leq
6k-3$) show that the condition in Theorem \ref{ThEr} is sharp.

We denote by $\bL_n^k$ and $\bN_n^k$ the graphs obtained from
$L_{n+1}^{k+1}$ and $N_{n+1}^{k+1}$, respectively, by deleting one
vertex of degree $n$, i.e., for $0\leq k\leq n/2-1$,
$$\bL^k_n=K_{k+1}+K_{n-k-1} \mbox{ and } \bN^k_n=K_k\vee(K_{n-2k-1}+(k+1)K_1).$$
In addition, we set
$$\mathcal{H}_n=\{G: K_{\lceil n/2\rceil-1,\lfloor n/2\rfloor+1}\subseteq G\subseteq
K_{\lceil n/2\rceil-1}\vee(\lfloor n/2\rfloor+1)K_1\}.$$ Note that
all graphs in $\mathcal{H}_n$ have the complements with the same
(signless Laplacian) spectral radius. Also note that every graph in
$\mathcal{H}_n$ is a subgraph of $N_n^{(n-1)/2}$ for odd $n$, and a
subgraph of $\bN_n^{n/2-1}$ for even $n$.

\begin{center}
\setlength{\unitlength}{0.8pt} \small
\begin{picture}(535,140)
\thicklines

\put(0,0){\put(70,80){\thinlines\circle{100}} \put(60,75){$K_{n-4}$}
\put(127.9,80){\put(-22.9,0){\circle*{4}} \put(0,31.6){\circle*{4}}
\put(0,-31.6){\circle*{4}} \put(37.1,19.5){\circle*{4}}
\put(37.1,-19.5){\circle*{4}} \put(0,-31.6){\line(0,1){63.2}}
\put(37.1,-19.5){\line(0,1){39}}
\qbezier(-22.9,0)(-11.5,15.8)(0,31.6)
\qbezier(-22.9,0)(7.1,9.8)(37.1,19.5)
\qbezier(37.1,19.5)(18.6,25.6)(0,31.6)
\qbezier(37.1,19.5)(18.6,-5.6)(0,-31.6)
\qbezier(-22.9,0)(-11.5,-15.8)(0,-31.6)
\qbezier(-22.9,0)(7.1,-9.8)(37.1,-19.5)
\qbezier(37.1,-19.5)(18.6,-25.6)(0,-31.6)
\qbezier(37.1,-19.5)(18.6,5.6)(0,31.6)} \put(90,10){$L_n^4$}}

\put(175,0){\put(70,80){\thinlines\circle{100}}
\put(60,75){$K_{n-4}$}
\put(135,80){\multiput(-30,-30)(0,20){4}{\circle*{4}}
\multiput(30,-30)(0,20){4}{\circle*{4}}
\put(-30,-30){\line(1,0){60}} \put(-30,-30){\line(3,1){60}}
\put(-30,-30){\line(3,2){60}} \put(-30,-30){\line(1,1){60}}
\put(-30,-10){\line(3,-1){60}} \put(-30,-10){\line(1,0){60}}
\put(-30,-10){\line(3,1){60}} \put(-30,-10){\line(3,2){60}}
\put(-30,10){\line(3,1){60}} \put(-30,10){\line(1,0){60}}
\put(-30,10){\line(3,-1){60}} \put(-30,10){\line(3,-2){60}}
\put(-30,30){\line(1,0){60}} \put(-30,30){\line(3,-1){60}}
\put(-30,30){\line(3,-2){60}} \put(-30,30){\line(1,-1){60}}}
\put(90,10){$N_n^4$}}

\put(350,0){\put(70,80){\thinlines\circle{100}}
\put(60,75){$K_{n-4}$}
\put(135,80){\multiput(-30,-20)(0,20){3}{\circle*{4}}
\multiput(30,-30)(0,20){4}{\circle*{4}}
\put(-30,-20){\line(6,-1){60}} \put(-30,-20){\line(6,1){60}}
\put(-30,-20){\line(2,1){60}} \put(-30,-20){\line(6,5){60}}
\put(-30,0){\line(2,-1){60}} \put(-30,0){\line(6,-1){60}}
\put(-30,0){\line(6,1){60}} \put(-30,0){\line(2,1){60}}
\put(-30,20){\line(6,1){60}} \put(-30,20){\line(6,-1){60}}
\put(-30,20){\line(2,-1){60}} \put(-30,20){\line(6,-5){60}}}
\put(90,10){$\bN_n^3$}}

\end{picture}

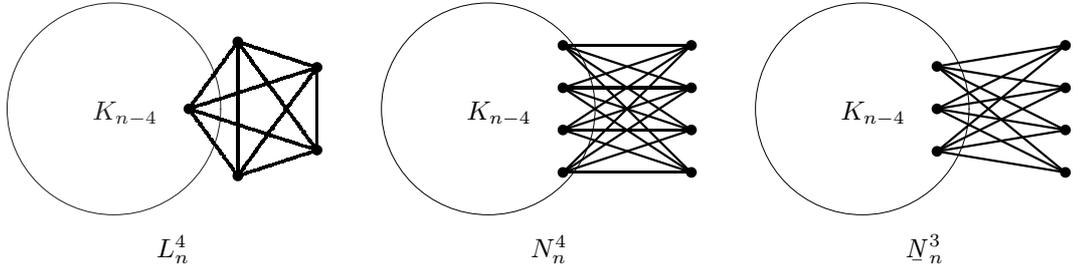
\small Fig. 1. Graphs $L_n^4$, $N_n^4$ and $\bN_n^3$.
\end{center}

Fiedler and Nikiforov \cite{FN} published their
important work on spectral conditions for Hamilton cycles and Hamilton
paths, which stimulated many subsequent researches on this topic.

\begin{theorem}[Fiedler and Nikiforov \cite{FN}]\label{ThFiNi}
Let $G$ be a graph of order $n$. \\
(1) If $\rho(G)\geq n-2$, then $G$ is traceable unless $G=\bN_n^0$.\\
(2) If $\rho(G)>n-2$, then $G$ is Hamiltonian unless $G=N^1_{n}$.
\end{theorem}

\begin{theorem}[Fiedler and Nikiforov \cite{FN}]\label{ThFiNi'}
Let $G$ be a graph of order $n$.\\
(1) If $\rho(\overline{G})\leq\sqrt{n-1}$, then $G$ is traceable
unless $G=\bL^0_{n}$. \\
(2) If $\rho(\overline{G})\leq\sqrt{n-2}$, then $G$ is Hamiltonian
unless $G=L^1_{n}$.
\end{theorem}

Fiedler and Nikiforov's theorems can be seen as spectral analogues
of Ore's theorem. Motivated by this fact, our first aim of this
paper is to give spectral analogues of Erd\H{o}s' theorem, i.e., to
replace the ¡°edge number condition¡± by ¡°spectral condition¡±
(together with minimum degree condition) to guarantee the existence
of Hamilton cycles (Hamilton paths) in graphs. Our first problem can
be stated as follows.

\begin{problem}\label{PrrhoGGC}
Among all non-Hamiltonian graphs (non-traceable graphs) $G$ of order
$n$ with $\delta(G)\geq k$, to determine the values of $\max\rho(G)$
and $\min\rho(\overline{G})$, respectively.
\end{problem}

The above problem follows some recent trends in extremal graph
theory, and contributes to a new but energetic studied area called
spectral extremal graph theory. For a comprehensive survey on this
area, we refer the reader to \cite{N2011} by Nikiforov.

Besides Theorems \ref{ThFiNi} and \ref{ThFiNi'}, there are also some
other works related to Problem \ref{PrrhoGGC}, see
\cite{LSX,LLT,NG}. However, a complete solution to
the problem is unknown till now. Our partial solution to Problem \ref{PrrhoGGC}
is as follows.

\begin{theorem}\label{ThrhoG}
Let $k$ be an integer, and $G$ be a graph of order $n$.\\
(1) If $\delta(G)\geq k\geq 0$ and $\rho(G)\geq
\rho(\bN_n^k)$, where $n\geq \max\{6k+10,(k^2+7k+8)/2\}$, then $G$ is traceable unless $G=\bN_n^k$; \\
(2) If $\delta(G)\geq k\geq 1$ and $\rho(G)\geq \rho(N_n^k)$, where
$n\geq\max\{6k+5,(k^2+6k+4)/2\}$, then $G$ is Hamiltonian unless $G=N_n^k$.
\end{theorem}

We completely determine the values of $\min\rho(\overline{G})$ in
Problem \ref{PrrhoGGC}.

\begin{theorem}\label{ThrhoGC}
Let $k$ be an integer, and $G$ be a graph of order $n$.\\
(1) If $\delta(G)\geq k\geq 0$, $n\geq 2k+2$ and
$\rho(\overline{G})\leq \rho(\overline{\bL_n^k})$, then $G$ is
traceable unless $G=\bL_n^k$, or $n=2k+2$ and
$G\in \mathcal{H}_n$; \\
(2) If $\delta(G)\geq k\geq 1$, $n\geq 2k+1$ and
$\rho(\overline{G})\leq\rho(\overline{L^k_n})$, then $G$ is
Hamiltonian unless $G=L_n^k$, or $n=2k+1$ and $G\in\mathcal{H}_n$.
\end{theorem}

For the signless Laplacian, Zhou \cite{Z}, Nikiforov
\cite{N2002}, Yu and Fan \cite{YF} and Liu et al.
\cite{LSX} gave some sufficient conditions for Hamilton
cycles or Hamilton paths in terms of signless Laplacian spectral
radii of a graph and its complement. We list the following result
which is closely related to our topic.

\begin{theorem}[Yu and Fan \cite{YF}]\label{ThYuFa}
Let $G$ be a graph of order $n\geq 6$.\\
(1) If $q(G)\geq 2n-4$, then $G$ is traceable unless
$G=\bN_n^0$.\\
(2) If $q(G)>2n-4$, then $G$ is Hamiltonian unless $G=N^1_{n}$.
\end{theorem}

In \cite{YF}, the bound of $n\geq 6$ is missed,
and in fact there are counterexamples of small order, namely
$K_{1,3}$ for traceability, and $K_{1,1,3}$ for Hamiltonicity. This
tiny flaw has already been pointed out in \cite{LSX} by Liu et al.

Motivated by Problem \ref{PrrhoGGC} and Theorem \ref{ThYuFa}, we have the following problem:

\begin{problem}\label{PrqG}
Among all non-Hamiltonian graphs (non-traceable graphs) $G$ of order
$n$ with $\delta(G)\geq k$, to determine the values of $\max q(G)$ and
$\min q(\overline{G})$, respectively.
\end{problem}

Our partial answer to Problem \ref{PrqG} is as follows.

\begin{theorem}\label{ThqG}
Let $k$ be an integer, and $G$ be a graph of order $n$.\\
(1) If $\delta(G)\geq k\geq 0$ and $q(G)\geq
q(\bN_n^k)$, where $n\geq\max\{6k+10,(3k^2+9k+8)/2\}$, then $G$ is traceable unless $G=\bN_n^k$; \\
(2) If $\delta(G)\geq k\geq 1$ and $q(G)\geq q(N_n^k)$,
where $n\geq\max\{6k+5,(3k^2+5k+4)/2\}$, then $G$ is
Hamiltonian unless $G=N_n^k$.
\end{theorem}

Nikiforov mentioned a result on the signless Laplacian spectral
radius of the complement of a graph and Hamiltoncity, see \cite[Section 3.8]{N2011}
for details.
\subsection{Hamiltonicity of balanced bipartite graphs}

Let $G$ be a bipartite graph with partite sets $\{X,Y\}$. We use
$\widehat{G}$ to denote the \emph{quasi-complement} of $G$, i.e.,
the graph with vertex set $V(\widehat{G})=V(G)$ and for any $x\in X$
and $y\in Y$, $xy\in E(\widehat{G})$ if and only if $xy\notin E(G)$.
The bipartite graph $G$ is called \emph{balanced} if $|X|=|Y|$. Note
that every Hamiltonian bipartite graph is balanced.

Our second aim of this paper is to find spectral analogues of Moon
and Moser's theorem, which is a bipartite analogue of Erd\H{o}s'
theorem and given as follows.

\begin{theorem}[Moon and Moser \cite{MM}]\label{ThMoMo}
Let $G$ be a balanced bipartite graph of order $2n$ with
$\delta(G)\geq k$, where $1\leq k\leq n/2$. If
$$e(G)>\max\left\{n(n-k)+k^2,n(n-\left\lfloor\frac{n}{2}\right\rfloor)+{\left\lfloor\frac{n}{2}\right\rfloor}^2\right\},$$
then $G$ is Hamiltonian.
\end{theorem}

Moon and Moser \cite{MM} also pointed out that a balanced
bipartite graph $G$ of order $2n$ is Hamiltonian if $\delta(G)>n/2$.

Let $B^k_n$ ($1\leq k\leq n/2$) be the graph obtained from $K_{n,n}$
by deleting all edges in its one subgraph $K_{n-k,k}$. Note that
$e(B_n^k)=n(n-k)+k^2$ and $B_n^k$ is not Hamiltonian. This type of
graphs shows the edge number condition in Theorem \ref{ThMoMo} is
sharp. We denote by $\mathcal{B}_n^k$ ($1\leq k\leq n/2$) the set of
balanced bipartite graphs in which each graph is obtained from a
bipartite graph $H$ with two partite sets $\{X,Y\}$ of size $k$ and
$n-k$, respectively, by adding $k$ additional vertices each of which
is adjacent to every vertex in $X$, and $n-k$ additional vertices
each of which is adjacent to every vertex in $Y$. Note that $B_n^k$
is the graph in $\mathcal{B}_n^k$ with the largest edge number. (In
this case, $H$ is a complete bipartite graph.)

\begin{center}
\setlength{\unitlength}{1.2pt} \small
\begin{picture}(190,70)
\thicklines

\multiput(20,10)(0,10){6}{\circle*{4}}
\multiput(70,10)(0,10){6}{\circle*{4}}
\multiput(120,20)(0,10){4}{\circle*{4}}
\multiput(170,20)(0,10){4}{\circle*{4}}

\put(20,10){\line(1,0){50}} \put(20,10){\line(5,1){50}}
\put(20,10){\line(5,2){50}} \put(20,10){\line(5,3){50}}
\put(20,10){\line(5,4){50}} \put(20,10){\line(1,1){50}}
\put(20,20){\line(1,0){50}} \put(20,20){\line(5,1){50}}
\put(20,20){\line(5,2){50}} \put(20,20){\line(5,3){50}}
\put(20,20){\line(5,4){50}} \put(20,20){\line(5,-1){50}}
\put(20,30){\line(1,0){50}} \put(20,30){\line(5,1){50}}
\put(20,30){\line(5,2){50}} \put(20,30){\line(5,3){50}}
\put(20,30){\line(5,-2){50}} \put(20,30){\line(5,-1){50}}
\put(20,40){\line(1,0){50}} \put(20,40){\line(5,1){50}}
\put(20,40){\line(5,2){50}} \put(20,40){\line(5,-3){50}}
\put(20,40){\line(5,-2){50}} \put(20,40){\line(5,-1){50}}
\put(20,50){\line(1,0){50}} \put(20,50){\line(5,1){50}}
\put(20,50){\line(5,-4){50}} \put(20,50){\line(5,-3){50}}
\put(20,50){\line(5,-2){50}} \put(20,50){\line(5,-1){50}}
\put(20,60){\line(1,0){50}} \put(20,60){\line(1,-1){50}}
\put(20,60){\line(5,-4){50}} \put(20,60){\line(5,-3){50}}
\put(20,60){\line(5,-2){50}} \put(20,60){\line(5,-1){50}}

\put(70,10){\line(5,1){50}} \put(70,10){\line(5,2){50}}
\put(70,10){\line(5,3){50}} \put(70,10){\line(5,4){50}}
\put(70,20){\line(5,1){50}} \put(70,20){\line(5,2){50}}
\put(70,20){\line(5,3){50}} \put(70,20){\line(1,0){50}}
\put(70,30){\line(5,1){50}} \put(70,30){\line(5,2){50}}
\put(70,30){\line(5,-1){50}} \put(70,30){\line(1,0){50}}
\put(70,40){\line(5,1){50}} \put(70,40){\line(5,-2){50}}
\put(70,40){\line(5,-1){50}} \put(70,40){\line(1,0){50}}
\put(70,50){\line(5,-3){50}} \put(70,50){\line(5,-2){50}}
\put(70,50){\line(5,-1){50}} \put(70,50){\line(1,0){50}}
\put(70,60){\line(5,-3){50}} \put(70,60){\line(5,-2){50}}
\put(70,60){\line(5,-1){50}} \put(70,60){\line(5,-4){50}}

\put(120,20){\line(5,1){50}} \put(120,20){\line(5,2){50}}
\put(120,20){\line(5,3){50}} \put(120,20){\line(1,0){50}}
\put(120,30){\line(5,1){50}} \put(120,30){\line(5,2){50}}
\put(120,30){\line(5,-1){50}} \put(120,30){\line(1,0){50}}
\put(120,40){\line(5,1){50}} \put(120,40){\line(5,-2){50}}
\put(120,40){\line(5,-1){50}} \put(120,40){\line(1,0){50}}
\put(120,50){\line(5,-3){50}} \put(120,50){\line(5,-2){50}}
\put(120,50){\line(5,-1){50}} \put(120,50){\line(1,0){50}}

\end{picture}

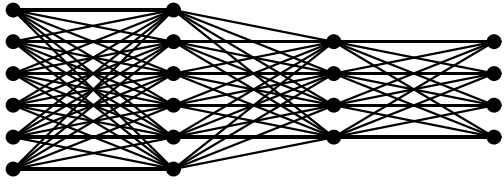
\small Fig. 2. The Graph $B_{10}^4$.
\end{center}

We remark that all graphs in $\mathcal{B}_n^k$ have the
quasi-complements of the same (signless Laplacian) spectral radius,
and for any (spanning) subgraph $G$ of $B_n^k$,
$\rho(\widehat{G})=\rho(\widehat{B_n^k})$
(resp. $q(\widehat{G})=q(\widehat{B_n^k})$) if and only if
$G\in\mathcal{B}_n^k$.

In this subsection, we consider a problem similar to Problems
\ref{PrrhoGGC} and \ref{PrqG} for balanced bipartite graphs.

\begin{problem}\label{PrrhoBGBGCqBGBGC}
Among all non-Hamiltonian balanced bipartite graphs $G$ of order
$2n$ with $\delta(G)\geq k$, to determine the values $\max\rho(G),
\min\rho(\widehat{G}), \max q(G)$ and $\min q(\widehat{G})$,
respectively.
\end{problem}

There are some results related to this problem, see
\cite{LSX,LLT}. Our partial solution to Problem
\ref{PrrhoBGBGCqBGBGC} is given as follows. The two special graphs
$\varGamma_1$ and $\varGamma_2$ are shown in Fig. 3 \cite[Fig.1]{FJP}.

\begin{theorem}\label{ThrhoBGBGCqBG}
Let $G$ be a balanced bipartite graph of order $2n$ and of minimum
degree $\delta(G)\geq k\geq 1$. \\
(1) If $n\geq (k+1)^2$ and $\rho(G)\geq\rho(B^k_n)$, then $G$ is
Hamiltonian unless $G=B^k_n$.\\
(2) If $n\geq (k+1)^2$ and $q(G)\geq q(B^k_n)$, then $G$ is
Hamiltonian unless $G=B^k_n$.\\
(3) If $n\geq 2k$ and $\rho(\widehat{G})\leq\rho(\widehat{B^k_n})$,
then $G$ is Hamiltonian unless $G\in\mathcal{B}_n^k$, or
$G=\varGamma_1$ or $\varGamma_2$ for $n=4$ and $k=2$.
\end{theorem}

\begin{center}
\begin{picture}(190,80)
\thicklines

\put(0,0){\multiput(20,30)(20,0){4}{\circle*{4}}
\multiput(20,70)(20,0){4}{\circle*{4}}
\multiput(80,30)(-20,0){3}{\line(0,1){40}}
\put(20,30){\line(1,2){20}} \put(20,30){\line(1,1){40}}
\put(20,30){\line(3,2){60}} \put(20,70){\line(1,-2){20}}
\put(20,70){\line(1,-1){40}} \put(20,70){\line(3,-2){60}}
\put(45,10){$\varGamma_1$} }

\put(90,0){\multiput(20,30)(20,0){4}{\circle*{4}}
\multiput(20,70)(20,0){4}{\circle*{4}}
\multiput(80,30)(-20,0){4}{\line(0,1){40}}
\put(20,30){\line(1,2){20}} \put(20,30){\line(1,1){40}}
\put(20,30){\line(3,2){60}} \put(20,70){\line(1,-2){20}}
\put(20,70){\line(1,-1){40}} \put(20,70){\line(3,-2){60}}
\put(45,10){$\varGamma_2$} }

\end{picture}

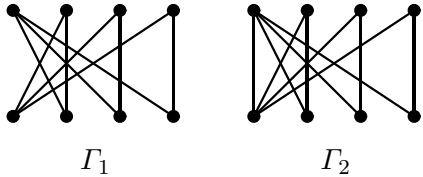
\small Fig. 3. Graphs $\varGamma_1$ and $\varGamma_2$.
\end{center}

For the signless Laplacian spectral radius of the quasi-complement
of a balanced bipartite graph, we have the following result. Note
that one cannot get a better bound on $q(\widehat{G})$ even if
one adds the minimum degree condition in Theorem \ref{ThqBGC}.

\begin{theorem}\label{ThqBGC}
Let $G$ be a balanced bipartite graph of order $2n$. If
$q(\widehat{G})\leq n$, then $G$ is Hamiltonian unless
$G\in\bigcup_{k=1}^{\lfloor n/2\rfloor}\mathcal{B}_n^k$, or
$G=\varGamma_1$ or $\varGamma_2$ for $n=4$.
\end{theorem}

\section{Preliminaries}

In this section, we will list our main tools. The first two
subsections contain useful structural theorems for general graphs
and for balanced bipartite graphs, respectively. The last subsection
includes some lower and upper bounds involving (signless Laplacian)
spectral radii of graphs.

\subsection{Structural lemmas for graphs}

The closure theory introduced by Bondy and Chv\'atal
\cite{BC} is a powerful tool for Hamiltonicity of graphs.
Let $G$ be a graph of order $n$. The \emph{closure} of $G$, denoted
by $cl(G)$, is the graph obtained from $G$ by recursively joining
pairs of nonadjacent vertices whose degree sum is at least $n$ until
no such pair remains. Bondy and Chv\'atal \cite{BC} proved
that the closure of $G$ is uniquely determined.

\begin{theorem}[Bondy and Chv\'{a}tal \cite{BC}]\label{ThBoCh}
A graph $G$ is Hamiltonian if and only if $cl(G)$ is Hamiltonian.
\end{theorem}

Let $G$ be a graph and $H$ be a subgraph of $G$.
For any vertex $v\in V(G)$, we define $N_{H}(v)=N(v)\cap V(H)$
and $d_{H}(v)=|N_{H}(v)|$. A graph $G$ is \emph{closed} if $G=cl(G)$, i.e., if any two
nonadjacent vertices of $G$ have degree sum less than $n(G)$. Now we
prove a lemma on the clique number of closed graphs.

\begin{lemma}\label{LeomegaG}
Let $G$ be a closed graph of order $n\geq 6k+5$, where $k\geq 1$. If
$$e(G)>\binom{n-k-1}{2}+(k+1)^2,$$
then $\omega(G)\geq n-k$.
\end{lemma}

\begin{proof}
A vertex of $G$ is called \emph{heavy} if it has degree at least
$n/2$. Since $G$ is closed, any two heavy vertices are adjacent in
$G$. Let $C$ be the vertex set of a maximum clique of $G$ containing all heavy
vertices and let $H=G-C$. Let $t=|C|$.

Suppose first that $1\leq t\leq n/3+k+1$. Then for every $v\in V(H)$,
we have $d_C(v)\leq t-1$ and $d(v)\leq (n-1)/2$. Note that
$$e(G[C])={t\choose 2} \mbox{ and } e(H)+e(V(H),C)=\frac{\sum_{v\in V(H)}d(v)+\sum_{v\in V(H)}d_C(v)}{2}.$$
Thus
\begin{align*}
e(G)    & =e(G[C])+e(H)+e(V(H),C)\\
        & \leq{t\choose 2}+\frac{(t-1+(n-1)/2)(n-t)}{2}\\
        &=\frac{n+1}{4}t+\frac{n(n-3)}{4}\\
        &\leq\frac{n+1}{4}\left(\frac{1}{3}n+k+1\right)+\frac{n(n-3)}{4}\\
        &=\frac{1}{3}n^2+\left(\frac{1}{4}k-\frac{5}{12}\right)n+\frac{k+1}{4}\\
        &\leq{n-k-1 \choose 2}+(k+1)^2\\
        &<e(G),
\end{align*}
a contradiction.

Suppose now that $(n+3)/3+k<t\leq n-k-1$. Note that $d(v)\leq
n-t$ for every $v\in V(H)$ (for otherwise $v$ will be adjacent to
every vertex of $C$). Since
$$e(G[C])={t\choose 2} \mbox{ and } e(H)+e(V(H),C)\leq\sum_{v\in V(H)}d(v),$$
we have
\begin{align*}
e(G)    & =e(G[C])+e(H)+e(V(H),C)\\
        & \leq{t\choose 2}+(n-t)^2\\
        & =\frac{3}{2}t^2-\left(2n+\frac{1}{2}\right)t+n^2\\
        & \leq{n-k-1 \choose 2}+(k+1)^2\\
        &<e(G),
\end{align*}
also a contradiction.

So we conclude that $t\geq n-k$ and $\omega(G)\geq n-k$.
\end{proof}

Armed with Lemma 1, we prove the following lemma which refines
Erd\H{o}s' theorem (Theorem 1.2) in some sense.

\begin{lemma}\label{LeeGH}
Let $G$ be a graph of order $n\geq 6k+5$, where $k\geq 1$. If
$\delta(G)\geq k$ and
$$e(G)>\binom{n-k-1}{2}+(k+1)^2,$$ then $G$ is Hamiltonian unless
$G\subseteq L^k_n$ or $N^k_n$.
\end{lemma}

\begin{proof}
Let $G'=cl(G)$. If $G'$ is Hamiltonian, then so is $G$ by Theorem
\ref{ThBoCh}. Now we assume that $G'$ is not Hamiltonian. Note that
$\delta(G')\geq\delta(G)$ and $e(G')\geq e(G)$. By Lemma
\ref{LeomegaG}, $\omega(G')\geq n-k$. Let $C$ be a maximum clique of
$G'$ and $H=G'-C$.

We claim that $\omega(G')=n-k$. Suppose that $\omega(G')\geq n-k+1$.
Since $G'$ is not a clique, $V(H)\neq\emptyset$. Let $v$ be a vertex
in $H$. Note that $d_{G'}(u)\geq n-k$ for every $u\in C$ and
$d_{G'}(v)\geq \delta(G')\geq k$. This implies $v$ is adjacent to
every vertex of $C$ in $G'$, contradicting that $C$ is a maximum
clique of $G'$. So $\omega(G')=n-k$, as we claimed.

Note that every vertex in $C$ has degree at least $n-k-1$ in $G'$.
We say that a vertex in $C$ is a \emph{frontier vertex} if it has
degree at least $n-k$ in $G'$, i.e., it has at least one neighbor in
$H$. Let $F=\{u_1,u_2,\ldots,u_s\}$ be the set of frontier vertices.
From the fact that $G'$ is closed, we can see that every vertex in
$H$ has degree exactly $k$ in $G'$, and every vertex in $H$ is
adjacent to every frontier vertex in $G'$. Moreover, since
$|V(H)|=k$, we can see that $1\leq s\leq k$.

If $s=1$, then $H$ is a clique and $G'=L^k_n$; if $s=k$, then
$H$ is an independent set and $G'=N^k_n$. In both cases we have
$G\subseteq L^k_n$ or $N^k_n$. Now we assume that $2\leq s\leq k-1$.
Let $P$ be a Hamilton path of $G'[(C-F)\cup\{u_1,u_s\}]$
from $u_1$ to $u_s$.

Note that every vertex in $H$ has degree $k-s$ in $H$. By Dirac's
theorem \cite{D}, $H$ has a path of order at least $k-s+1$.
First we assume that $H$ has a path $P'$ of order $k-s+2$. Let
$x,x'$ be the two end-vertices of $P'$ and
$V(H-P')=\{v_1,\ldots,v_{s-2}\}$. Then $u_1v_1u_2v_2\cdots
u_{s-2}v_{s-2}u_{s-1}xP'x'u_sPu_1$ is a Hamilton cycle of $G'$, a
contradiction.

Now we assume that $H$ has no paths of order more than $k-s+1$. Let
$P'$ be a path of order $k-s+1$ in $H$, and $x,x'$ be the two end-vertices
of $P'$. Clearly $x$ has no neighbor in $V(H-P')$, which implies
that $xx'\in E(H)$. Since $H$ has no path longer than $P'$, every
vertex in $V(H-P')$ has no neighbor in $P'$, specially, $H-P'$ has
an edge $v_1v_2$. Let $V(H-P')=\{v_1,v_2,\ldots,v_{s-1}\}$. Then
$u_1v_1v_2u_2\cdots v_{s-1}u_{s-1}xP'x'u_sPu_1$ is a Hamilton cycle
of $G'$, also a contradiction.
\end{proof}

We also have an analogue of Lemma \ref{LeeGH} for traceable graphs.

\begin{lemma}\label{LeeGT}
Let $G$ be a graph of order $n\geq 6k+10$, where $k\geq 0$. If
$\delta(G)\geq k$ and
\[
e(G)>\binom{n-k-2}{2}+(k+1)(k+2),
\]
then $G$ is traceable unless $G\subseteq \bL_n^k$ or $\bN_n^k$.
\end{lemma}

\begin{proof}
Let $G'=G\vee K_1$. Note that $G$ is traceable if and only if $G'$
is Hamiltonian. We have $n(G')=n+1\geq 6(k+1)+5$, $\delta(G')\geq
k+1\geq 1$ and
\[
e(G')=e(G)+n>\binom{n-k-2}{2}+(k+1)(k+2)+n=\binom{n-k-1}{2}+(k+2)^2.
\]
By Lemma \ref{LeeGH}, $G'$ is Hamiltonian unless $G'\subseteq
L^{k+1}_{n+1}$ or $N^{k+1}_{n+1}$. Thus $G$ is traceable unless
$G\subseteq \bL_n^k$ or $\bN_n^k$.
\end{proof}

We will also use the following result. It was originally related to \cite{AC},
and was strengthened in \cite{FJH}. For details, see Theorem 3.1 in \cite{FJH}.

\begin{theorem}[Ainouche and Christofides \cite{AC}]\label{ThAiCh}
Let $G$ be a non-Hamiltonian graph. If $d(u)+d(v)\geq n-1$ for every
two nonadjacent vertices $u,v\in V(G)$, then either $G=L_n^k$ for
$1\leq k\leq(n-1)/2$, or $n$ is odd and $G\in\mathcal{H}_n$.
\end{theorem}

\subsection{Structural lemmas for balanced bipartite graphs}

Let $G$ be a balanced bipartite graph of order $2n$. The
\emph{bipartite closure} (or briefly, \emph{B-closure}) of $G$,
denoted by $cl_B(G)$, is the graph obtained from $G$ by recursively
joining pairs of nonadjacent vertices in different partite sets
whose degree sum is at least $n+1$ until no such pair remains.

\begin{theorem}[Bondy and Chv\'{a}tal \cite{BC}]\label{ThBoCh'}
A balanced bipartite graph $G$ is Hamiltonian if and only if
$cl_B(G)$ is Hamiltonian.
\end{theorem}

A balanced bipartite graph $G$ of order $2n$ is \emph{B-closed} if
$G=cl_B(G)$, i.e., if every two nonadjacent vertices in distinct
partite sets of $G$ have degree sum at most $n$. We have the
following result on B-closed balanced
bipartite graphs.

\begin{lemma}\label{LezBG}
Let $G$ be a B-closed balanced bipartite graph of order $2n$. If
$n\geq 2k+1$ for some $k\geq 1$ and $$e(G)>n(n-k-1)+(k+1)^2,$$ then
$G$ contains a complete bipartite graph of order $2n-k$.
Furthermore, if $\delta(G)\geq k$, then $K_{n,n-k}\subseteq G$.
\end{lemma}

\begin{proof}
Let $X,Y$ be the two partite sets of $G$. We denote by $h(X)$ and
$h(Y)$ the number of vertices in $X$ and $Y$, respectively, with
degree larger than $n/2$. Then $$nh(X)+\frac{1}{2}n(n-h(X))\geq
e(G).$$ Thus
$$h(X)\geq \frac{2e(G)}{n}-n\geq\frac{2n(n-k-1)+2(k+1)^2+2}{n}-n.$$
One can compute that $h(X)>k$ when $n\geq 2k+1$. Similarly we have
$h(Y)>k$. Clearly every vertex in $X$ with degree more than $n/2$
and every vertex in $Y$ with degree more than $n/2$ are adjacent.
This implies that $K_{k+1,k+1}\subseteq G$. Now let $t$ be the
maximum integer such that $K_{t,t}\subset G$. Thus $t\geq k+1$.

\begin{claim}
$t\geq n-k$.
\end{claim}

\begin{proof}
Suppose not. Then $k+1\leq t\leq n-k-1$. Let $X'\subset X$,
$Y'\subset Y$ such that $G[X'\cup Y']=K_{t,t}$. Recall that $G$ is
B-closed. If for any $x\in X\backslash X'$, there exists $y\in Y'$
such that $xy\notin E(G)$, then for any $x\in X\backslash X'$,
$d(x)\leq n-t$; if there exits $x\in X\backslash X'$, such that for
any $y\in Y'$, $xy\in E(G)$, then for any $y\in Y\backslash Y'$,
$d(y)\leq n-t$. Without loss of generality, assume that for any
$y\in Y\backslash Y'$, $d(y)\leq n-t$. Then
\begin{align*}
e(G)    &=e(X',Y')+e(X\backslash X',Y')+e(X,Y\backslash Y')\\
        &\leq t^2+t(n-t)+(n-t)^2\\
        &=nt+(n-t)^2\\
        &\leq n(n-k-1)+(k+1)^2\\
        &<e(G),
\end{align*}
a contradiction.
\end{proof}

Now let $s$ be a largest integer such that $K_{s,t}\subset G$. Thus
$s\geq t$.

\begin{claim}
$s+t\geq 2n-k$.
\end{claim}

\begin{proof}
Suppose not. Then $n-k\leq t\leq n-(k+1)/2$ and $t\leq s\leq
2n-k-t-1$. Without loss of generality, let $X'\subset X$,
$Y'\subset Y$, such that $G[X',Y']=K_{s,t}$. Then for any $x\in
X\backslash X'$, $d(x)\leq n-s$; and for any $y\in Y\backslash Y'$,
$d(y)\leq n-t$. Thus
\begin{align*}
e(G)    &\leq e(X',Y')+e(X\backslash X',Y)+e(X,Y\backslash Y')\\
        &\leq st+(n-s)^2+(n-t)^2\\
        &=s^2-(2n-t)s+n^2+(n-t)^2\\
        &\leq (2n-k-t-1)^2-(2n-t)(2n-k-t-1)+n^2+(n-t)^2\\
        &=t^2-(2n-k-1)t+n^2+(n-k-1)^2\\
        &\leq(n-k)^2-(2n-k-1)(n-k)+n^2+(n-k-1)^2\\
        &=n(n-k-1)+k^2+k+1\\
        &<e(G),
\end{align*}
a contradiction.
\end{proof}

By Claim 2, $K_{s,t}$ is a complete bipartite graph with order at
least $2n-k$. This completes the proof of the first part.

Suppose that $\delta(G)\geq k$. If $K_{n,n-k}\not\subseteq G$, then
$n-k+1\leq t\leq s\leq n$. Let $X'\subset X, Y'\subset Y$, such that
$G[X',Y']=K_{s,t}$. Then for any $x\in X\backslash X'$, $x$ is
adjacent to any vertex of $Y'$, this implies that $s=n$. Thus
$K_{n,n-k+1}\subseteq G$, a contradiction.
\end{proof}

\begin{lemma}\label{LeeBGH}
Let $G$ be a balanced bipartite graph of order $2n$. If
$\delta(G)\geq k\geq 1$, $n\geq 2k+1$ and
$$e(G)>n(n-k-1)+(k+1)^2,$$ then $G$ is Hamiltonian unless $G\subseteq B_n^k$.
\end{lemma}

\begin{proof}
Let $G'=cl_B(G)$. If $G'$ is Hamiltonian, then so is $G$ by Theorem
\ref{ThBoCh'}. Now we assume that $G'$ is not Hamiltonian. Note that
$\delta(G')\geq\delta(G)$ and $e(G')\geq e(G)$. By Lemma
\ref{LezBG}, $K_{n,n-k}\subseteq G'$. Let $t$ be the largest integer
such that $K_{n,t}\subseteq G$. Clearly $n-k\leq t<n$. Let $X,Y$ be
the partite sets of $G$, and $Y'\subset Y$ such that $G[X\cup Y']=K_{n,t}$.

We claim that $t=n-k$. Suppose that $t\geq n-k+1$. Note that every
vertex in $X$ has degree at least $t\geq n-k+1$ in $G'$ and every
vertex in $Y$ has degree at least $k$. This implies that $G'$ is a
complete bipartite graph, a contradiction. Thus $t=n-k$, as we claimed.

Note that every vertex in $X$ has degree at least $n-k$ in $G'$. We
say here that a vertex in $X$ is a \emph{frontier vertex} if it has
degree at least $n-k+1$ in $G'$, i.e., it has at least one neighbor
in $Y\backslash Y'$. From the fact that $G'$ is closed, we can see
that every vertex in $Y\backslash Y'$ has degree exactly $k$ in
$G'$, and every vertex in $Y\backslash Y'$ is adjacent to every
frontier vertex in $G'$. Thus there are exactly $k$ frontier
vertices in $X$ and $G'=B_n^k$. So $G\subseteq B_n^k$.
\end{proof}

The following result is a balanced bipartite graph version of Theorem \ref{ThAiCh}.

\begin{theorem}[Ferrara, Jacobson, and Powell \cite{FJP}]\label{ThFeJaPo}
Let $G$ be a non-Hamiltonian balanced bipartite graph. If
$d(x)+d(y)\geq n$ for every two nonadjacent vertices $x,y$ in
distinct partite sets, then either
$G\in\bigcup_{k=1}^{n/2}\mathcal{B}_n^k$, or $G=\varGamma_1$ or
$\varGamma_2$ for $n=4$.
\end{theorem}

\subsection{Spectral inequalities}

We will use the following spectral inequalities for graphs and
bipartite graphs.

\begin{theorem}[Nikiforov \cite{N2002}]\label{ThNi'}
Let $G$ be a graph of order $n$ with $\delta(G)\geq k$. Then
$$\rho(G)\leq\frac{k-1}{2}+\sqrt{2e(G)-nk+\frac{(k+1)^2}{4}}.$$
\end{theorem}

\begin{theorem}[Feng and Yu \cite{FY}]\label{ThFeYu}
Let $G$ be a graph of order $n$. Then $$q(G)\leq
\frac{2e(G)}{n-1}+n-2.$$
\end{theorem}

\begin{theorem}[Bhattacharya, Friedland, and Peled \cite{BFP}]\label{ThBhFrPe}
Let $G$ be a bipartite graph. Then
$$\rho(G)\leq \sqrt{e(G)}.$$
\end{theorem}

\begin{theorem}[Feng and Yu \cite{FY}, Yu and Fan \cite{YF}]\label{ThFYYF}
Let $G$ be a graph with non-empty edge set. Then $$q(G)\leq
\max\left\{d(u)+\frac{\sum_{v\in N(u)}d(v)}{d(u)}: u\in
V(G)\right\}.$$
\end{theorem}

\begin{theorem}\label{ThLiNi}
Let $G$ be a balanced bipartite graph of order $2n$. Then
$$q(G)\leq \frac{e(G)}{n}+n.$$
\end{theorem}

\begin{proof}
If $G$ is an edgeless graph, then it is trivially true. Now assume
$G$ contains at least one edge. Let $x$ be a vertex in $V(G)$
maximizing the right hand of the formula in Theorem \ref{ThFYYF}. By
Theorem \ref{ThFYYF},
\begin{align*}
n+\frac{e(G)}{n}-q(G)    &\geq \left(n+\frac{\sum_{y\in N(x)}d(y)}{n}\right)-\left(d(x)+\frac{\sum_{y\in N(x)}d(y)}{d(x)}\right)\\
        &=\left(n-d(x)\right)\left(1-\frac{\sum_{y\in N(x)}d(y)}{nd(x)}\right)\\
        &\geq 0.
\end{align*}
This completes the proof.
\end{proof}

The following two theorems can be proved similarly as Lemma 2.1 in
\cite{BZ} due to Berman and Zhang, and Theorem 2 in
\cite{AM} due to Anderson and Morley (see also Proposition 3.9.1 in \cite{BH}), respectively. We
omit the details of the proofs.

\begin{theorem}\label{ThBeZh}
Let $G$ be a graph with non-empty edge set. Then
$$\rho(G)\geq\min\{\sqrt{d(u)d(v)}: uv\in E(G)\}.$$
Moreover, if $G$ is connected, then equality holds if and only if
$G$ is regular or semi-regular bipartite.
\end{theorem}

\begin{theorem}\label{ThAnMo}
Let $G$ be a graph with non-empty edge set. Then
$$q(G)\geq\min\{d(u)+d(v): uv\in E(G)\}.$$
Moreover, if $G$ is connected, then the equality holds if and only if
$G$ is regular or semi-regular bipartite.
\end{theorem}

Let $G$ be a graph and $u,v\in V(G)$. We construct a new graph $G'$
in the following way: for every $w\in
N(u)\backslash(N(v)\cup\{v\})$, replace the edge $uw$ by a new edge
$vw$. The above operation of graphs, introduced by Kelmans
\cite{Ke}, is called the Kelmans operation. (See pp.44 in \cite{BH}.) Wu, Xiao, and Hong \cite{WXH}
proved that the spectral radius of a connected graph increases
under the Kelmans operation. For general graphs, the similar result was obtained by Csikv\'ari
\cite{C} later, independently. For connected graphs, a similar observation
also holds for the signless Laplacian spectral radius
under the Kelmans operation, see Feng \cite{F}.

\begin{theorem}\label{ThCs}
Let $G$ be a graph and $G'$ be a graph obtained from $G$ by a Kelmans
operation. Then \\
(1) (Wu, Xiao, and Hong \cite{WXH}, Csikv\'ari \cite{C}) $\rho(G')\geq\rho(G)$; and\\
(2) $q(G')\geq q(G)$.
\end{theorem}

\begin{proof}
For the convenience of readers, we write the detailed proof of (2) here.
Let $A$ and $A'$ be the adjacency matrices, and $D$
and $D'$ be the degree matrices, of $G$ and $G'$, respectively. Let
$(A+D)\bix=q(G)\bix$, where $\bix\geq 0$ and $\bix^T\bix=1$. For two
vertices $u$ and $v$ corresponding to the Kelmans operation, without
loss of generality, let $x_u\geq x_v$. Set
$W=N(v)\backslash(N(u)\cup\{u\})$. Then
\begin{align*}
q(G')   &\geq\bix^T(A'+D')\bix\\
        &=\bix^TA'\bix+\bix^TD'\bix\\
        &=\bix^TA\bix+2(x_u-x_v)\sum_{w\in W}x_w+\bix^TD\bix+(x_u^2-x_v^2)|W|\\
        &\geq \bix^T(A+D)\bix\\
        &=q(G).
\end{align*}
Thus the inequality holds.
\end{proof}

\begin{lemma}\label{LeCompare}
\begin{align*}
(1)\ &\rho(\bL_n^0)=\rho(K_{n-1})=n-2,  q(\bL_n^0)=q(K_{n-1})=2n-4,
\rho(\overline{\bL_n^0})=\rho(K_{1,n-1})=\sqrt{n-1}.
\end{align*}\\[-20mm]
\begin{align*}
(2)\ &\rho(L_n^1)>\rho(K_{n-1})=n-2,  q(L_n^1)>q(K_{n-1})=2n-4,
\rho(\overline{L_n^1})=\rho(K_{1,n-2})=\sqrt{n-2}.
\end{align*}\\[-20mm]
\begin{align*}
(3) \mbox{ For } k\geq 1,\ & \rho(\bN_n^k)>\rho(\bL_n^k)=\rho(K_{n-k-1})=n-k-2,\\
    & q(\bN_n^k)>q(\bL_n^k)=q(K_{n-k-1})=2n-2k-4, \mbox{ and}\\
    & \rho(\overline{\bN_n^k})\geq\rho(\overline{\bL_n^k})=\rho(K_{k+1,n-k-1})=\sqrt{(k+1)(n-k-1),}\ \mbox{with equality \ \ \ } \\
            & \mbox{only if } n \mbox{ is even and } k=n/2-1.
\end{align*}\\[-20mm]
\begin{align*}
(4) \mbox{ For } k\geq 2,\ & \rho(N_n^k)>\rho(L_n^k)>\rho(K_{n-k})=n-k-1,\\
    & q(N_n^k)>q(L_n^k)>q(K_{n-k})=2n-2k-2, \mbox{ and}\\
    & \rho(\overline{N_n^k})\geq\rho(\overline{L_n^k})=\rho(K_{k,n-k-1})=\sqrt{k(n-k-1),}\ \mbox{with equality only if \ \ \ } \\
            & n \mbox{ is odd and } k=(n-1)/2.
\end{align*}\\[-20mm]
\begin{align*}
(5) \mbox{ For } k\geq 1,\ & \rho(B_n^k)>\rho(K_{n,n-k})=\sqrt{n(n-k)},\ q(B_n^k)>q(K_{n,n-k})=2n-k, \mbox{\ \ \ \ \ \ \ \ \ }\\
    & \rho(\widehat{B_n^k})=\rho(K_{k,n-k})=\sqrt{k(n-k)}, \ q(\widehat{B_n^k})=q(K_{k,n-k})=n.
\end{align*}
\end{lemma}

\begin{proof}
(1)--(5) other than (4) can be deduced by the fact that the (signless
Laplacian) spectral radius decreases after deleting an edge in
connected graphs.

Now we prove (4).
It is not difficult to see that if we do $k-1$ Kelmans operations
on $L_n^k$ ($k\geq 2$), then we can obtain a proper subgraph of
$N_n^k$. By Theorem \ref{ThCs}, $\rho(N_n^k)>\rho(L_n^k)$ and
$q(N_n^k)>q(L_n^k)$. In the following we will prove
$\rho(\overline{N_n^k})\geq\rho(\overline{L_n^k})=\sqrt{k(n-k-1)}$
for $k\geq 2$, with equality only if $n$ is odd and $k=(n-1)/2$.

Note that $\overline{N_n^k}=K_k\vee(n-2k)K_1$. From Theorem 2.8 in
\cite{CDS}, we have the formula
$$\rho(K_k\vee
(n-2k)K_1)=\frac{k-1+\sqrt{4k(n-k)-(3k-1)(k+1)}}{2}.$$ Thus
\begin{align*}
\left(2\rho(\overline{L_n^k})-(k-1)\right)^2    &=4k(n-k-1)+(k-1)^2-4(k-1)\sqrt{k(n-k-1)}\\
        &\leq 4k(n-k-1)+(k-1)^2-4k(k-1)\\
        &=4k(n-k)-(3k-1)(k+1)\\
        &=\left(2\rho(\overline{N_n^k})-(k-1)\right)^2.
\end{align*}
Since $\rho(\overline{N_n^k})\geq (k-1)/2$, $\rho(\overline{N_n^k})\geq\rho(\overline{L_n^k})$.
Note that the equality holds only if either $k=1$ (which is not in our
assumption) or $n$ is odd and $k=(n-1)/2$.

The proof is complete.
\end{proof}

\section{Proofs of the theorems}

\textbf{Proof of Theorem \ref{ThrhoG}.}

(1) By Lemma \ref{LeCompare} and Theorem \ref{ThNi'},
$$n-k-2<\rho(G)\leq\frac{k-1}{2}+\sqrt{2e(G)-nk+\frac{(k+1)^2}{4}}.$$
Thus, when $n\geq(k^2+7k+8)/2$, we have
\begin{align*}
e(G)>\frac{n^2-(2k+3)n+2(k+1)^2}{2}\geq\binom{n-k-2}{2}+(k+1)(k+2).
\end{align*}
By Lemma \ref{LeeGT}, $G$ is traceable
or $G\subseteq \bL_n^k$ or $\bN_n^k$. But if $G\subseteq \bL_n^k$
for $k\geq 1$ or $G\subset \bN_n^k$, then $\rho(G)<\rho(\bN_n^k)$, a
contradiction. Thus $G=\bN_n^k$.

(2) By Lemma \ref{LeCompare} and Theorem \ref{ThNi'},
$$n-k-1<\rho(G)\leq\frac{k-1}{2}+\sqrt{2e(G)-nk+\frac{(k+1)^2}{4}}.$$
Thus, when $n\geq(k^2+6k+4)/2$, we have
\begin{align*}
e(G)>\frac{n^2-(2k+1)n+k(2k+1)}{2}\geq\binom{n-k-1}{2}+(k+1)^2.
\end{align*}
By Lemma \ref{LeeGH}, $G$ is Hamiltonian
or $G\subseteq L_n^k$ or $N_n^k$. But if $G\subseteq L_n^k$ for
$k\geq 2$ or $G\subset N_n^k$, then $\rho(G)<\rho(N_n^k)$, a
contradiction. Thus $G=N_n^k$. The proof is complete. {\hfill$\Box$}

\vskip\baselineskip\noindent\textbf{Proof of Theorem \ref{ThrhoGC}.}

(1) The proof is based on the assertion (2), which will be proved
later. Let $G'=G\vee K_1$. Then $n(G')=n+1$,
$\delta(G')=\delta(G)+1\geq k+1$ and
$$\rho(\overline{G'})=\rho(\overline{G})\leq\rho(\overline{\bL_n^k})=\sqrt{(k+1)(n-k-1)}=\rho(\overline{L_{n+1}^{k+1}}).$$
By (2), $G'$ is Hamiltonian unless $G'=L_{n+1}^{k+1}$, or
$n+1=2(k+1)+1$ and $G\in\mathcal{H}_{n+1}$. Thus $G$ is traceable
unless $G=\bL_n^k$, or $n=2k+2$ and $G\in\mathcal{H}_n$.

(2) Let $G'=cl(G)$. If $G'$ is Hamiltonian, then so is $G$ by
Theorem \ref{ThBoCh}. Now we assume that $G'$ is not Hamiltonian.
Note that $G'$ is closed. Thus every two nonadjacent vertices $u,v$
have degree sum at most $n-1$, i.e.,
$$d_{\overline{G'}}(u)+d_{\overline{G'}}(v)\geq 2(n-1)-(n-1)=n-1.$$
Note that every non-trivial component of $\overline{G'}$ has a
vertex of degree at least $(n-1)/2$ and hence of order at least
$(n+1)/2$. This implies that $\overline{G'}$ has exactly one
nontrivial component. Since $d(u)\geq k$ and $d(v)\geq k$, we have
$d_{\overline{G'}}(u)\leq n-k-1$ and $d_{\overline{G'}}(v)\leq
n-k-1$. Thus $d_{\overline{G'}}(u)\geq k$ and
$d_{\overline{G'}}(v)\geq k$. This implies that
$$d_{\overline{G'}}(u)d_{\overline{G'}}(v)\geq d_{\overline{G'}}(u)(n-1-d_{\overline{G'}}(u))\geq k(n-k-1),$$
with equality if and only if (up to symmetry),
$d_{\overline{G'}}(u)=k$ and $d_{\overline{G'}}(v)=n-k-1$.

By Lemma \ref{LeCompare} and Theorem \ref{ThBeZh},
$$\sqrt{k(n-k-1)}\geq\rho(\overline{G})\geq\rho(\overline{G'})
\geq\min_{uv\in
E(\overline{G'})}\sqrt{d_{\overline{G'}}(u)d_{\overline{G'}}(v)}\geq\sqrt{k(n-k-1)}.$$
This implies that $\rho(\overline{G'})=\sqrt{k(n-k-1)}$ and there is
an edge $uv\in E(\overline{G'})$ such that $d_{\overline{G'}}(u)=k$
and $d_{\overline{G'}}(v)=n-k-1$. Let $H$ be the component of
$\overline{G'}$ containing $uv$. By Theorem \ref{ThBeZh}, $H$ is
regular or semi-regular bipartite. This implies that every two
nonadjacent vertices in $G'$ have degree sum $n-1$. By Theorem
\ref{ThAiCh}, $G'=L_n^k$ or $n=2k+1$ and $G'\in\mathcal{H}_n$. It is
easy to find that for any (spanning) subgraph of $L_n^k$ or any
(spanning) subgraph of a graph in $\mathcal{H}_n$ (when $n=2k+1$),
if it is not $L_n^k$ or is not in $\mathcal{H}_n$, then it has the
complement with spectral radius greater than
$\rho(\overline{L_n^k})$. Thus $G=L_n^k$ or $n=2k+1$ and
$G\in\mathcal{H}_n$. The proof is complete. {\hfill$\Box$}

\vskip\baselineskip\noindent\textbf{Proof of Theorem \ref{ThqG}.}

(1) By Lemma \ref{LeCompare} and Theorem \ref{ThFeYu},
$$2n-2k-4<q(G)\leq\frac{2e(G)}{n-1}+n-2.$$ Thus, when $n\geq(3k^2+9k+8)/2$, we have
\begin{align*}
e(G)    &>\frac{(n-1)(n-2k-2)}{2}\\
        &=\frac{n^2-(2k+3)n+2(k+1)}{2}\\
        &\geq\binom{n-k-2}{2}+(k+1)(k+2).
\end{align*}
By Lemma \ref{LeeGT}, $G$ is traceable
or $G\subseteq \bL_n^k$ or $\bN_n^k$. But if $G\subseteq \bL_n^k$
for $k\geq 1$ or $G\subset \bN_n^k$, then $q(G)<q(\bN_n^k)$, a
contradiction. Thus $G=\bN_n^k$.

(2) By Lemma \ref{LeCompare} and Theorem \ref{ThFeYu},
$$2n-2k-2<q(G)\leq\frac{2e(G)}{n-1}+n-2.$$
Thus, when $n\geq(3k^2+5k+4)/2$, we have
\begin{align*}
e(G)    &>\frac{(n-1)(n-2k)}{2}\\
        &=\frac{n^2-(2k+1)n+2k}{2}\\
        &\geq\binom{n-k-1}{2}+(k+1)^2.
\end{align*}
By Lemma \ref{LeeGH}, $G$ is
Hamiltonian or $G\subseteq L_n^k$ or $N_n^k$. But if $G\subseteq
L_n^k$ for $k\geq 2$ or $G\subset N_n^k$, then $q(G)<q(N_n^k)$, a
contradiction. Thus $G=N_n^k$. The proof is complete. {\hfill$\Box$}

\vskip\baselineskip\noindent\textbf{Proof of Theorem
\ref{ThrhoBGBGCqBG}.}

(1) By Lemma \ref{LeCompare} and Theorem \ref{ThBhFrPe},
$$\sqrt{n(n-k)}<\rho(G)\leq\sqrt{e(G)}.$$
Thus, we obtain
\begin{align*}
e(G)>n(n-k)\geq n(n-k-1)+(k+1)^2
\end{align*}
when $n\geq (k+1)^2$. By Lemma \ref{LeeBGH}, $G$ is Hamiltonian or $G\subseteq B_n^k$. But
if $G\subset B_n^k$, then $\rho(G)<\rho(B_n^k)$, a contradiction.
Thus $G=B_n^k$.

(2) By Lemma \ref{LeCompare} and Theorem \ref{ThLiNi},
$$2n-k<q(G)\leq \frac{e(G)}{n}+n.$$
Thus, there holds
\begin{align*}
e(G)>n(n-k)\geq n(n-k-1)+(k+1)^2
\end{align*}
when $n\geq (k+1)^2$. By Lemma \ref{LeeBGH}, $G$ is Hamiltonian or $G\subseteq B_n^k$. But
if $G\subset B_n^k$, then $q(G)<q(B_n^k)$, a contradiction. Thus
$G=B_n^k$.

(3) Let $G'=cl_B(G)$. If $G'$ is Hamiltonian, then so is $G$ by
Theorem \ref{ThBoCh'}. Now we assume that $G'$ is not Hamiltonian.
Note that $G'$ is B-closed. Thus every two nonadjacent vertices
$x\in X,y\in Y$ in distinct partite sets $X,Y$ have degree sum at
most $n$, i.e.,
$$d_{\widehat{G'}}(x)+d_{\widehat{G'}}(y)\geq 2n-n=n.$$
Since $\delta(G')\geq\delta(G)\geq k$, we can see that
$d_{\widehat{G'}}(x)\leq n-k$ and $d_{\widehat{G'}}(y)\leq n-k$.
Thus $d_{\widehat{G'}}(x)\geq k$ and $d_{\widehat{G'}}(y)\geq k$.
This implies that $$d_{\widehat{G'}}(x)d_{\widehat{G'}}(y)\geq d_{\widehat{G'}}(x)(n-d_{\widehat{G'}}(x))\geq
k(n-k),$$ with equality if and only if (up to symmetry)
$d_{\widehat{G'}}(x)=k$ and $d_{\widehat{G'}}(y)=n-k$. By Theorem
\ref{ThBeZh},
$$\sqrt{k(n-k)}\geq\rho(\widehat{G})\geq\rho(\widehat{G'})\geq\min_{xy\in
E(\widehat{G'})}\sqrt{d_{\widehat{G'}}(x)d_{\widehat{G'}}(y)}\geq\sqrt{k(n-k)}.$$
This implies that $\rho(\widehat{G'})=\sqrt{k(n-k)}$ and there is an
edge $xy\in E(\widehat{G'})$ such that $d_{\widehat{G'}}(x)=k$ and
$d_{\widehat{G'}}(y)=n-k$. Let $H$ be the component of
$\widehat{G'}$ containing $xy$. By Theorem \ref{ThBeZh}, $H$ is a
semi-regular bipartite graph, say, with partite sets $X'\subseteq
X$ and $Y'\subseteq Y$, and for every vertex $x'\in X'$,
$d(x')=d(x)=k$, and for every vertex $y'\in Y'$, $d(y')=d(y)=n-k$.
If $H=K_{k,n-k}$, then $G'\subseteq B_n^k$. If $H\neq K_{k,n-k}$,
then $n(H)>n$. Note that every nontrivial component of
$\widehat{G'}$ has order at least $(n+1)$. Thus $H$ is the unique
non-trivial component of $\widehat{G'}$. This implies that every two
nonadjacent vertices in distinct partite sets in $G'$ have degree
sum at least $n$. By Theorem \ref{ThFeJaPo}, $G'\in\mathcal{B}_n^k$
or $G'=\varGamma_1$ or $\varGamma_2$ for $n=4$ and $k=2$. In any
case, we can see that $G\subseteq B_n^k$ or $G\subseteq\varGamma_1$
or $\varGamma_2$ for $n=4$ and $k=2$. Note that for any (spanning)
subgraph of $\varGamma_1$, $\varGamma_2$ or $B_n^k$, if is not
$\varGamma_1$, or $\varGamma_2$, or a graph in $\mathcal{B}_n^k$,
then it has the quasi-complement with spectral radius greater than
$\rho(\widehat{B_n^k})$. Thus $G\in\mathcal{B}_n^k$ or
$G=\varGamma_1$ or $\varGamma_2$ for $n=4$ and $k=2$. The proof is
complete. {\hfill$\Box$}

\vskip\baselineskip\noindent\textbf{Proof of Theorem \ref{ThqBGC}.}

Let $G'=cl_B(G)$. If $G'$ is Hamiltonian, then so is $G$ by Theorem
\ref{ThBoCh'}. Now assume that $G'$ is not Hamiltonian. Similarly
as the proof of Theorem \ref{ThrhoBGBGCqBG}, for every two
nonadjacent vertices $x\in X,y\in Y$ in distinct partite sets
$X,Y$ of $G'$, we get
$$d_{\widehat{G'}}(x)+d_{\widehat{G'}}(y)\geq n.$$
By Theorem \ref{ThAnMo}, we have
$$n\geq q(\widehat{G})\geq q(\widehat{G'})\geq\min_{xy\in E(\widehat{G'})}(d(x)+d(y))\geq n,$$
This implies that $q(\widehat{G'})=n$ and there is an edge $xy\in
E(\widehat{G'})$ such that
$d_{\widehat{G'}}(x)+d_{\widehat{G'}}(y)=n$. Let $H$ be the
component of $\widehat{G'}$ containing $xy$. By Theorem
\ref{ThAnMo}, $H$ is a semi-regular bipartite graph, say, with
partite sets $X'\subseteq X$ and $Y'\subseteq Y$. If $H$ is a
complete bipartite graph $K_{k,n-k}$ for some $k$, then $G'\subseteq
B_n^k$. Otherwise, $n(H)>n$. Note that every nontrivial component of
$\widehat{G'}$ has order at least $n$. Thus $H$ is the unique
non-trivial component of $\widehat{G'}$. This implies that every two
nonadjacent vertices in distinct partite sets in $G$ have degree
sum at least $n$. By Theorem \ref{ThFeJaPo},
$G'\in\bigcup_{k=1}^{n/2}\mathcal{B}_n^k$, or $G'=\varGamma_1$ or
$\varGamma_2$ for $n=4$. In any case, we can see that $G\subseteq
B_n^k$ for $1\leq k\leq n/2$, or $G\subseteq\varGamma_1$ or
$\varGamma_2$ for $n=4$. Note that every (spanning) subgraph of
$\varGamma_1$, $\varGamma_2$ or $B_n^k$, $1\leq k\leq n/2$, if is
not $\varGamma_1$ or $\varGamma_2$, or a graph in
$\mathcal{B}_n^k$, then has the quasi-complement with signless Laplacian
spectral radius greater than $n$. Thus
$G\in\bigcup_{k=1}^{n/2}\mathcal{B}_n^k$, or $G=\varGamma_1$ or
$\varGamma_2$ for $n=4$. The proof is complete. {\hfill$\Box$}

\section*{Acknowledgements}

Part of this work has been done, while the second author was
visiting University of Science and Technology of China on March of
2015. He is grateful for the outstanding hospitality of Professor
Jie Ma. He is also grateful to Dr. Jun Ge for many invaluable
discussions.

The authors are very grateful to two anonymous referrers for their invaluable
suggestions which largely improve the presentation of the paper, especially
for one referee who points out that for connected graphs, Theorem 2.12 originally appeared in \cite{WXH,F},
and recommends the related paper \cite{HZ} to them.


\begin{thebibliography}{10}
\bibitem{AC}
A. Ainouche, N. Christofides, Conditions for the existence of Hamiltonian circuits in graphs based on vertex degrees,
\emph{J. London Math. Soc.} (2) {\bf 32} (1985) 385--391.

\bibitem{AM}
W. Anderson, T. Morley, Eigenvalues of the Laplacian of
a graph, \emph{Linear and Multilinear Algebra} {\bf 18} (2) (1985) 141--145.

\bibitem{BZ}
A. Berman, X.-D. Zhang, On the spectral radius of graphs with cut
vertices, \emph{J. Combin. Theory, Ser. B} {\bf 83} (2) (2001) 233--240.

\bibitem{BFP}
A. Bhattacharya, S. Friedland, U.N. Peled, On the first eigenvalue
of bipartite graphs, \emph{Electron. J. Combin.} {\bf 15} (2008) R144.

\bibitem{BC}
J.A. Bondy, V. Chv\'atal, A method in graph theory, \emph{Discrete Math.} {\bf 15}
(2) (1976) 111--135.

\bibitem{BH}
A. Brouwer, W. Haemers, Spectra of Graphs, Universitext, Springer, New York, 2012.

\bibitem{C}
P. Csikv\'ari, On a conjecture of V. Nikiforov, \emph{Discrete Math.} {\bf 309}
(13) (2009) 4522--4526.

\bibitem{CDS}
D. Cvetkovi\'c, M. Doob, H. Sachs, Spectra of Graphs: Theory and application,
Academic Press, New York-London, 1980.

\bibitem{D}
G.A. Dirac, Some theorems on abstract graphs, \emph{Proc. London Math.
Soc.} {\bf 2} (3) (1952) 69--81.

\bibitem{E}
P. Erd\H{o}s, Remarks on a paper of P\'{o}sa, \emph{Magyar Tud. Akad.
Mat. Kut. Int. K\"{o}zl.} {\bf 7} (1962) 227--229.

\bibitem{F}
L.-H. Feng, The signless Laplacian spectral radius for bicyclic graphs with $k$ pendant vertices,
\emph{Kyungpook Math. J.} {\bf 50} (1) (2010) 109--116.

\bibitem{FY}
L.-H Feng, G.-H Yu, On three conjectures involving the signless
Laplacian spectral radius of graphs, \emph{Publ. Inst. Math.} {\bf 85} (2009)
35--38.

\bibitem{FJH}
M. Ferrara, M. Jacobson, A.K. Harlan, Hamiltonian cycles avoiding sets of edges in a graph,
\emph{Australas. J. Combin.} {\bf 48} (2010), 191--203.

\bibitem{FJP}
M. Ferrara, M. Jacobson, J. Powell, Characterizing degree-sum
maximal nonhamiltonian bipartite graphs, \emph{Discrete Math.} {\bf 312} (2)
(2012) 459--461.

\bibitem{FN}
M. Fiedler, V. Nikiforov, Spectral radius and Hamiltonicity of
graphs, \emph{Linear Algebra Appl.} {\bf 432} (9) (2010) 2170--2173.

\bibitem{HZ}
Y. Hong, X.-D. Zhang, Sharp upper and lower bounds for largest eigenvalue of the
Laplacian matrices of trees, \emph{Discrete Math.} {\bf 296} 187--197.

\bibitem{Ka}
R.M. Karp, Reducibility among combinatorial problems, in: Complexity
of Computer Computations, Plenum Press, New York, 1972, 85--103.

\bibitem{Ke}
A.K. Kelmans, On graphs with randomly deleted edges, \emph{Acta Math.
Acad. Sci. Hungar.} {\bf 37} (1981) 77--88.

\bibitem{LSX}
R.-F Liu, W.C. Shiu, J. Xue, Sufficient spectral conditions on
Hamiltonian and traceable graphs, \emph{Linear Algebra Appl.} {\bf 467} (2015)
254--266.

\bibitem{LLT}
M. Lu, H.-Q Liu, F. Tian, Spectral radius and Hamiltonian graphs,
\emph{Linear Algebra Appl.} {\bf 437} (7) (2012) 1670--1674.


\bibitem{MM}
J. Moon, L. Moser, On Hamiltonian bipartite graphs, \emph{Israel J. Math.}
{\bf 1} (3) (1963) 163--165.

\bibitem{N2002}
V. Nikiforov, Some inequalities for the largest eigenvalue of a
graph, \emph{Combin. Probab. Comput.} {\bf 11} (2) (2002) 179--189.

\bibitem{N2011}
V. Nikiforov, Some new results in extremal graph theory, in: Surveys
in Combinatorics 2011, Cambridge University Press, 2011, 141--181.

\bibitem{NG}
B. Ning, J. Ge, Spectral radius and Hamiltonian properties of
graphs, \emph{Linear Multilinear Algebra} {\bf 63} (8) (2015) 1520--1530.

\bibitem{O}
O. Ore, Arc coverings of graphs, \emph{Ann. Mat. Pura Appl.} {\bf 55} (4) (1961) 315--321.

\bibitem{WXH}
B.-F. Wu, E.-L. Xiao, Y. Hong, The spectral radius of trees on $k$ pendant
vertices, \emph{Linear Algebra Appl} {\bf 395} (2005), 343--349.

\bibitem{YF}
G.-D. Yu, Y.-Z. Fan, Spectral conditions for a graph to be
Hamilton-connected, \emph{Applied Mechanics and Materials} 336--338 (2013)
2329--2334.

\bibitem{Z}
B. Zhou, Signless Laplacian spectral radius and Hamiltonicity,
\emph{Linear Algebra Appl.} {\bf 432} (2--3) (2010) 566--570.
\end{thebibliography}
\end{document}